\documentclass{amsart}%
\usepackage{amsfonts}
\usepackage{amsmath,amsthm}
\usepackage{amssymb}
\usepackage{graphicx}
\usepackage{amsmath}%
\providecommand{\U}[1]{\protect\rule{.1in}{.1in}}
\newtheorem{theorem}{Theorem}
\theoremstyle{plain}
\newtheorem{acknowledgement}{Acknowledgement}

\newtheorem{definition}{Definition}

\newtheorem{lemma}{Lemma}

\newtheorem{proposition}{Proposition}

\numberwithin{equation}{section}
\begin{document}
\title[Generalized convexity and applications]{Generalized convexity and the existence of finite time blow-up solutions for
an evolutionary problem}
\author{Constantin P. Niculescu}
\address{Department of Mathematics, University of Craiova, Craiova RO-200585, Romania}
\email{cniculescu47@yahoo.com}
\author{Ionel Roven\c{t}a}
\address{Department of Mathematics, University of Craiova, Craiova RO-200585, Romania}
\email{roventaionel@yahoo.com}
\thanks{Corresponding author: Roven\c{t}a Ionel}

\begin{abstract}
In this paper we study a class of nonlinearities for which a nonlocal
parabolic equation with Neumann-Robin boundary conditions, for $p$-Laplacian,
has finite time blow-up solutions.

\end{abstract}
\maketitle

\section{Introduction}

It is a well known fact that convexity plays an important role in the
different parts of mathematics, including the study of boundary value
problems. The aim of our paper is to introduce a new class of generalized
convex functions and to illustrate its usefulness in establishing a sufficient
condition for the existence of finite time blow-up solutions for the
evolutionary problem%

\begin{equation}
\left\{
\begin{array}
[c]{ccc}%
u_{t}-\Delta_{p}u=f(|u|)-\frac{1}{m(\Omega)}%
{\displaystyle\int\nolimits_{\Omega}}
f(|u|)\,dx & \text{in} & \Omega\\
&  & \\
|\nabla u|^{p-2}\frac{\partial u}{\partial n}=0 & \text{on} & \partial
\Omega\,,
\end{array}
\right.  \label{ft1}%
\end{equation}
with the initial conditions
\begin{equation}
u(x,0)=u_{0}(x)\;\text{on}\;\Omega,\;\text{where}\;%
{\displaystyle\int\nolimits_{\Omega}}
u_{0}\,dx=0. \label{ft2}%
\end{equation}

Here $\Omega\subset\mathbb{R}^{N}$ is a bounded regular domain of class
$C^{2}$, $f:[0,\infty)\mapsto\lbrack0,\infty)$ is a locally Lipschitz
function, $m(\Omega)$ represents the Lebesgue measure of the domain $\Omega,$
and $\Delta_{p}=div(|\nabla u|^{p-2}\nabla u)$, for $p\geq2$, is the
$p$-Laplacian operator.

The particular case where $p=2$ was recently considered by Soufi, Jazar and
Monneau \cite{SJM}, and Jazar and Kiwan \cite{JK} (under the assumption that
$f$ is a power function of the form $f(u)=u^{\alpha},$ with $\alpha>1),$ and
also by the present authors \cite{NR2010} (for $f$ belonging to a larger class
of nonlinearities).

The problems of type (\ref{ft1}) \& (\ref{ft2}) arise naturally in mechanics,
biology and population dynamics. See \cite{JB}, \cite{LAC}, \cite{RSCC},
\cite{YC} and \cite{JFMG}. For example, if we consider a couple or a mixture
of two equations of the above type, the resulting problem describes the
temperatures of two substances which constitute a combustible mixture, or
represents a model for the behavior of densities of two diffusion biological
species which interact each other.

\section{Generalized convexity of order $\alpha$}

According to the classical Hermite-Hadamard inequality, the mean value of a
continuous convex function $f:[a,b]\rightarrow\mathbb{R}$ lies between the
value of $f$ at the midpoint of the interval $[a,b]$ and the arithmetic mean
of the values of $f$ at the endpoints of this interval, that is, \emph{\ }
\begin{equation}
\,f\left(  \frac{a+b}{2}\right)  \leq\frac{1}{b-a}\,\int_{a}^{b}%
\,f(x)\,dx\leq\frac{f(a)+f(b)}{2}\,. \tag{$HH$}%
\end{equation}
Moreover, each side of this double inequality characterizes convexity in the
sense that a real-valued continuous function $f$ defined on an interval
$I$\ is convex if its restriction to each compact subinterval $[a,b]\subset I$
verifies the left hand side of $(HH)$ (equivalently, the right hand side of
$(HH))$. See \cite{BP} and \cite{NP2006} for details.

In what follows we will be interested in a class of generalized convex
functions motivated by the right hand side of the Hermite-Hadamard inequality.

\begin{definition}
A real-valued function $f$ defined on an interval $[a,\infty)$ belongs to the
class $GC_{\alpha}$ $($for some $\alpha>0)$, if it is continuous, nonnegative,
and%
\begin{equation}
\frac{1}{\alpha+1}f(t)\geq\frac{1}{t-a}\int_{a}^{t}f(x)\,dx\text{\quad for
}t\text{ large enough.} \label{GCa}%
\end{equation}

\end{definition}

Using calculus, one can see easily that the condition (\ref{GCa}) is
equivalent to the fact that the ratio%
\begin{equation}
\frac{\frac{1}{t-a}\int_{a}^{t}f(x)\,dx}{\left(  t-a\right)  ^{\alpha}}
\label{defC}%
\end{equation}
is nondecreasing for $t$ bigger than a suitable value $A\geq a.$ In turn, this
implies that the mean value $\frac{1}{t-a}\int_{a}^{t}f(x)\,dx$ has a
polynomial growth at infinity.

According to the Hermite-Hadamard inequality, every nonnegative, continuous
and convex function $f:[a,\infty)\rightarrow\mathbb{R}$ with $f(a)=0$ belongs
to the class $GC_{1}.$ The converse is not true because the membership of a
function $f:[a,\infty)\rightarrow\mathbb{R}$ to the class $GC_{\alpha}$ yields
only an asymptotic inequality of the form%
\[
\frac{1}{\alpha+1}f(t)+\frac{\alpha}{\alpha+1}f(a)\geq\frac{1}{t-a}\int
_{a}^{t}f(x)\,dx\text{\quad for }t\text{ large enough.}%
\]

If $g\in C^{1}([0,\infty))$ and $g$ is nondecreasing, then the function
$f(x)=g(x)(x-a)^{\alpha}$ belongs to the class $CG_{\alpha}\left(
[0,\infty)\right)  ,$ whenever $\alpha>0$. In fact,%
\begin{align*}
\frac{1}{t-a}\int_{a}^{t}f(x)dx  &  =\frac{(t-a)^{\alpha}}{\alpha+1}%
g(t)-\frac{1}{t-a}\int_{a}^{t}g^{\prime}(x)\frac{\left(  x-a\right)
^{\alpha+1}}{\alpha+1}\;dx\\
&  \leq\frac{1}{\alpha+1}f(t).
\end{align*}

As a consequence, $(x+\sin x)x$ provides an example of function of class
$GC_{1}$ on $[0,\infty)$ which is not convex.

No positive constant can be a function of class $GC_{\alpha}$ for any
$\alpha>0$.

Also, the restriction of a function $f:[a,\infty)\rightarrow\mathbb{R}$ of
class $GC_{\alpha}$ to a subinterval $[b,\infty)$ is not necessarily a
function of class $GC_{\alpha}.$

In the sequel we will describe some other classes of functions of class
$GC_{\alpha}.$

The following concept of generalized convexity is due to S. Varosanec
\cite{SV} and generalizes the usual convexity, $s$-convexity, the
Godunova--Levin functions and $P$-functions.

\begin{definition}
Suppose that $h:[0,1]\rightarrow\mathbb{R}$ is a function such that
$h(\lambda)+h(1-\lambda)\geq1$ for all $\lambda\in\lbrack0,1].$ A nonnegative
function $f$ defined on an interval $I$ is called $h$-convex if
\begin{equation}
f(\lambda x+(1-\lambda)y)\leq h(\lambda)f(x)+h(1-\lambda)f(y). \label{sconv1}%
\end{equation}
whenever $\lambda\in\lbrack0,1]$,$\,$\ and $x,y\in I$.
\end{definition}

\begin{proposition}
\label{PropV}Suppose that $f$ is a nonnegative continuous function defined on
an interval $[a,\infty)$ such that the following two conditions are fulfilled:

$i)$ $f(a)=0;$

$ii)$ $f$ is $h$-$convex$ with respect to a function $h$ with $\int_{0}%
^{1}h(\lambda)\;d\lambda\leq\frac{1}{\alpha+1}$, for some $\alpha>0.$

Then $f$ belongs to the class $GC_{\alpha}$.
\end{proposition}

\begin{proof}
In fact,
\begin{align*}
\frac{1}{t-a}\int_{a}^{t}f(x)\;dx  &  =\int_{0}^{1}f((1-\lambda)a+\lambda
t)\;d\lambda\\
&  \leq f(t)\int_{0}^{1}h(\lambda)\;d\lambda+f(a)\int_{0}^{1}h(1-\lambda
)\;d\lambda\\
&  \leq\frac{1}{\alpha+1}f(t).
\end{align*}

\end{proof}

An important class of nonlinearities in partial differential operators theory
is that of regularly varying functions, introduced by Karamata in \cite{K1930}.

\begin{definition}
A positive measurable function $f$ defined on interval $[a,\infty)$ (with
$a\geq0)$ is said to be regularly varying at infinity, of index $\sigma
\in\mathbb{R}$ \emph{(}abbreviated, $f\in RV_{\infty}(\sigma)$\emph{)},
provided that
\[
\lim_{x\rightarrow\infty}\frac{f(tx)}{f(x)}=t^{\sigma}\quad\text{for}%
\;\text{all}\;t>0.
\]

\end{definition}

All functions of index $\sigma$ are of the form
\[
f(x)=x^{\sigma}\exp\left(  a(x)+\int_{0}^{x}\frac{\varepsilon(s)}{s}ds\right)
,
\]
where $a(x)$ and $\varepsilon(x)$ are bounded and measurable, $a(x)\rightarrow
\alpha\in\mathbb{R}$ and $\varepsilon(x)\rightarrow0$ as $x\rightarrow\infty.$
In particular, so are
\[
x^{\sigma}\log x,\,\ x^{\sigma}\log\log x,~x^{\sigma}\exp\left(  \frac{\log
x}{\log\log x}\right)  ,\,\ x^{\sigma}\exp\left(  \left(  \log x\right)
^{1/3}\left(  \cos\left(  \log x\right)  ^{1/3}\right)  \right)  .
\]

See \cite{B} for details.

Semilinear problems with nonlinearities in the class of regularly varying
functions have been studied by many people. See the paper by C\^{\i}rstea and
R\u{a}dulescu \cite{CR2} and the references therein.

\begin{proposition}
\label{PropRV}If $f\in RV_{\infty}(\sigma)$ with $\sigma>0,$ then%
\[
\lim_{x\rightarrow\infty}\frac{F(x)}{xf(x)}=\frac{1}{\sigma+1},
\]
where%
\begin{equation}
F(x):=\int_{0}^{x}f(s)\,ds. \label{primf}%
\end{equation}
As a consequence, if $f$ is also continuous, then $f$ is of class $GC_{\alpha
},$ whenever $\alpha\in(0,\sigma).$
\end{proposition}

\begin{proof}
To prove this, consider the change of variable $s=tx$ which yields
\[
F(x)=\int_{0}^{x}f(s)\,ds=\int_{0}^{1}xf(tx)\,dt.
\]

The continuity of $f$ and the fact that $f\in RV_{\infty}(\sigma)$ assure the
existence of a $\delta>0$ such that for every $x>\delta$ we have
\[
\frac{f(tx)}{f(x)}\leq t^{\sigma}+1,
\]
whence the integrability of the function $t\rightarrow\frac{f(tx)}{f(x)}$ on
$[0,1]$. Then%
\begin{align*}
\lim_{x\rightarrow\infty}\frac{F(x)}{xf(x)}  &  =\lim_{x\rightarrow\infty}%
\int_{0}^{1}\frac{f(tx)}{f(x)}\,dt\\
&  =\int_{0}^{1}\lim_{x\rightarrow\infty}\frac{f(tx)}{f(x)}\,dt=\int_{0}%
^{1}t^{\sigma}dt=\frac{1}{\sigma+1}.
\end{align*}
The commutation of the limit with the integral is motivated by the Lebesgue
dominated convergence theorem.
\end{proof}

Another important class of nonlinearities which appear in connection with the
study of boundary blow-up problems for elliptic equations is the class of
functions satisfying the Keller-Osserman condition. See \cite{Rad},
\cite{DDGR}, \cite{IR2008} and \cite{NR2010}.

\begin{definition}
\label{KO} A nonnegative and nondecreasing function $f\in C^{1}([0,\infty))$
with $f(0)=0$ satisfies the generalized Keller-Osserman condition of order
$p>1$ if
\begin{equation}
\int_{1}^{\infty}\frac{1}{(F(t))^{1/p}}\;dt<\infty, \label{f2}%
\end{equation}
where $F$ is the primitive of $f$ given by the formula (\ref{primf}).
\end{definition}

If $f\in RV_{\infty}(\sigma+1)$ with $\sigma+2>p>1$ is a nondecreasing and
continuous function, then $F\in RV_{\infty}(\sigma+2)$ and $F^{-1/p}\in
RV_{\infty}((-\sigma-2)/p)$. Since $(-\sigma-2)/p<-1$, we infer that
$F^{-1/p}\in L^{1}([1,\infty))$ and thus $f$ satisfies the generalized
Keller-Osserman condition.

It is worth to notice that the function $\exp(t)$ is not regularly varying at
infinity though satisfies the generalized Keller-Osserman condition and
belongs also to any class $GC_{\alpha}$ with $\alpha>0$.

Necessarily, if a function $f$ satisfies the generalized Keller-Osserman
condition of order $p>1$, then%
\begin{equation}
\lim_{t\rightarrow\infty}\frac{F(t)}{t^{p}}=\infty, \label{growthF}%
\end{equation}
while $\frac{F(t)}{t^{p}}$ may be (or may be not) a monotonic function.

If $\frac{F(t)}{t^{p}}$ is nondecreasing for some $p>2$, then the function $f$
belongs to the class $GC_{p-1}$. In particular, this is the case of the
function $f(t)=pt^{p-1}\log(t+1)+\frac{t^{p}}{t+1}$ (whose primitive is
$F(t)=t^{p}\log(t+1)).$ Notice that this function does not satisfy the
generalized Keller-Osserman condition of order $p.$

We end this section by discussing the connection Definition 1 with a class of
functions due to W. Orlicz.

\begin{definition}
\label{Orl}An $\emph{N}$\emph{-function} is any function $M:[0,\infty
)\rightarrow\mathbb{R}$ of the form%
\[
M(x)=\int_{0}^{x}p(t)\,dt,
\]
where $p$ is nondecreasing and right continuous, $p(0)=0,$ $p(t)>0$ for $t>0,$
and $\lim_{t\rightarrow\infty}p(t)=\infty.$

An $N$\emph{-}function\emph{ }$M$ satisfies the $\Delta_{2}$-\emph{condition}
if there exist constants $k>0$ and $x_{0}\geq0$ such that%
\[
M(2x)\leq kM(x)\text{\quad for all }x\geq x_{0}.
\]

\end{definition}

Any $N$\emph{-}function\emph{ }$M$ is convex and plays the following properties:

\medskip

$N1)$ $M(0)=0$ and $M(x)>0$ for $x>0;$

$N2)$ $\frac{M(x)}{x}\rightarrow0$ as $x\rightarrow0$ and $\frac{M(x)}%
{x}\rightarrow\infty$ as $x\rightarrow\infty.$

$\medskip$

Two examples of $N$-functions which satisfy the $\Delta_{2}$-condition are
$\frac{x^{p}}{p}$ (for $p\geq1)$ and $t(\log t)^{+}$.

The $N$-functions which satisfy the $\Delta_{2}$-condition are instrumental in
the theory of Orlicz spaces (which extend the $L^{p}(\mu)$ spaces). Their
theory is available in many books, such as \cite{KR} and \cite{RR}, and has
important applications to interpolation theory \cite{BS} and Fourier analysis
\cite{ZYG}.

According to \cite{KR}, page 23, the constant $k$ which appears in the
formulation of $\Delta_{2}$-condition is always greater than or equal to 2.

\begin{proposition}
Every $N$-function $M:[0,\infty)\rightarrow\mathbb{R}$ which satisfies the
$\Delta_{2}$-condition belongs to the class $GC_{\alpha},$ whenever $\alpha
\in(0,2\log_{2}k)$.
\end{proposition}

\begin{proof}
Since $M$ is nondecreasing,%
\[
M(tx)=M(2^{\log_{2}t}x)\leq M(2^{\left[  \log_{2}t\right]  +1}x),
\]
and taking into account the $\Delta_{2}$-condition we infer that
\begin{align*}
M(tx)  &  \leq M(x)k^{\left[  \log_{2}t\right]  +1}\leq M(x)k^{\log_{2}t+1}\\
&  \leq M(x)t^{2\log_{2}k},
\end{align*}
for $x\ $big enough and $t\geq2.$ Hence,%
\begin{align*}
\int_{0}^{t}M(x)dx  &  =\int_{0}^{1}tM(ts)ds\\
&  \leq\int_{0}^{1}tM(t)s^{2\log_{2}k}ds=\frac{1}{2\log_{2}k+1}tM(t)
\end{align*}
and the proof is done.
\end{proof}

\section{An application to the existence of finite time blow-up solutions}

This section is devoted to the existence of finite time blow-up solutions of
the evolutionary $p$-Laplacian problem%
\begin{equation}
u_{t}-\Delta_{p}u=f(|u|)-\frac{1}{m(\Omega)}%
{\displaystyle\int\nolimits_{\Omega}}
f(|u|)\,dx\text{\quad in }\Omega\label{eveq}%
\end{equation}
with Neumann-Robin boundary values,%
\begin{equation}
|\nabla u|^{p-2}\frac{\partial u}{\partial n}=0\text{\quad on }\partial
\Omega\,, \label{ncond}%
\end{equation}
and the initial conditions%
\begin{equation}
u(x,0)=u_{0}(x)\;\text{on}\;\Omega,\;\text{where}\;%
{\displaystyle\int\nolimits_{\Omega}}
u_{0}\,dx=0. \label{inval}%
\end{equation}

As was mentioned in the introduction, we restrict ourselves to the case where
$\Omega\subset\mathbb{R}^{N}$ is a bounded regular domain of class $C^{2}$,
and $f:[0,\infty)\mapsto\lbrack0,\infty)$ is a locally Lipschitz function;
$m(\Omega)$ represents the Lebesgue measure of the domain $\Omega,$ and
$\Delta_{p}$, for $p\geq2$, is the $p$-Laplacian operator.

The purpose of this section, is to extend a natural energetic criterion for
the blow-up in finite time of solutions of $(3.1)-(3.3).$ Our proof relies on
the same idea used by Jazar and Kiwan \cite{JK} in the case where $p=2$ and
$f$ is a power function.

We start by noticing that each solution $u$ of the problem above has the
property
\[
\int_{\Omega}u\,dx=0
\]
because the integral in the right hand side of (\ref{eveq}) is $0$ and
\begin{align*}
\frac{d}{dt}\Big(\int_{\Omega}u\,dx\Big)  &  =\int_{\Omega}u_{t}%
\,dx=\int_{\Omega}\Delta_{p}u\,dx\\
&  =\int_{\Omega}div(|\nabla u|^{p-2}\nabla u)\,dx=0.
\end{align*}
Hence, by the initial condition (\ref{inval}), we have $\int_{\Omega}u\,dx=0$.

Next, it is easy to see that for $p>1$ the energy
\[
E(u(t))=\int_{\Omega}\Big(\frac{1}{p}|\nabla u|^{p}-\int_{0}^{u}%
f(|\tau|)\,d\tau\Big)dx,
\]
of any solution $u$ of our evolutionary problem is nonincreasing in time. In
fact,%
\begin{align*}
\frac{dE(u(t))}{dt}  &  =\int_{\Omega}\left(  |\nabla u|^{p-2}\nabla
u_{t}\nabla u-u_{t}f(|u|)\right)  \,dx\\
&  =\int_{\partial\Omega}\frac{\partial u}{\partial n}|\nabla u|^{p-2}%
u_{t}\,d\sigma-\int_{\Omega}u_{t}\Delta_{p}u\,dx-\int_{\Omega}u_{t}%
\,f(|u|)\,dx\\
&  =-\int_{\Omega}u_{t}(\Delta_{p}u+f(|u|))\,dx=-\int_{\Omega}u_{t}^{2}\,dx,
\end{align*}
and by integrating both sides over $[0,t]$ we obtain the formula
\begin{equation}
E(u(t))=E(u_{0})-\int_{0}^{t}\int_{\Omega}u_{t}^{2}\,dxdt,\quad\text{for\ all}%
\;t>0. \label{ft5}%
\end{equation}

According to this formula, if the initial energy $E(u_{0})$ is nonpositive,
then $E(u(t))$ is nonpositive for all $t>0.$ In the case of generalized convex
functions of order $\alpha$, with $\alpha>\frac{1}{p-1}$ we have%
\begin{equation}
C\int_{\Omega}uf(|u|)\,dx\geq\int_{\Omega}\int_{0}^{u}f(|t|)\,dtdx\geq\frac
{1}{p}\int_{\Omega}|\nabla u|^{p}, \label{ftinegmod}%
\end{equation}
where $C=\frac{1}{1+\alpha}\in(0,\frac{p-1}{p})$.

\begin{theorem}
\emph{(}The energetic criterion for blow-up in finite time, case\textbf{
}$p\geq2$\emph{)} \label{ftthm1} Assume that $f:[0,\infty)\mapsto
\lbrack0,\infty)$ is a locally Lipschitz function belonging to the class
$GC_{a}$, with $\alpha>\frac{1}{p-1}$, and let $u$ be a solution of the
problem $(3.1)-(3.3)$ corresponding to an initial data $u_{0}\in
C(\overline{\Omega})$, $u_{0}$ not identically zero.

If $E(u_{0})\leq0,$ then $u,$ as a function of $t,$ cannot be in $L^{\infty
}((0,T);L^{2}(\Omega))$ for all $T>0$. In other word, there is $T>0$ such that%
\begin{equation}
\underset{t\rightarrow T-}{\lim\sup}\left\Vert u(t)\right\Vert _{L^{2}}%
=\infty. \label{blowup}%
\end{equation}

\end{theorem}

Notice that the condition $E(u_{0})\leq0$ in Theorem \ref{ftthm1} is also
necessary for the blow-up in finite time (of the $L^{2}$ norm of $u(t)).$ In
fact, (\ref{blowup}) forces that
\[
\inf\left\{  E(u(t)):0<t<T\right\}  =-\infty.
\]
This can be argued by contradiction. If $E(u(t))\geq-C_{0}$, for some
$C_{0}>0$, then the function%
\[
h(t):=\frac{1}{2}\int_{\Omega}u^{2}(x,t)\,dx
\]
verifies the condition
\begin{align*}
\frac{1}{2}h^{\prime}(t)  &  =\int_{\Omega}uu_{t}dx\leq\frac{1}{2}\int
_{\Omega}\left(  u^{2}+u_{t}^{2}\right)  dx\\
&  =\frac{1}{2}(h(t)-E^{\prime}(u(t))),
\end{align*}
which yields
\[
(h(t)+E(u(t))+C_{0})^{\prime}\leq h(t)\leq h(t)+E(u(t))+C_{0}.
\]
Therefore%
\[
h(t)\leq h(t)+E(u(t))+C_{0}\leq(h(0)+E(u_{0})+C_{0})e^{t},\;\text{for all
}t\in(0,T),
\]
and thus the $L^{2}$-norm of $u(t)$ is bounded.

The proof of Theorem \ref{ftthm1} needs a preparation.

\begin{lemma}
\label{ftlem2} Under the assumptions of Theorem \ref{ftthm1} the two auxiliary
functions%
\[
h(t):=\frac{1}{2}\int_{\Omega}u^{2}(x,t)\,dx\quad\text{and\quad}H(t):=\int
_{0}^{t}h(s)\,ds
\]
verify the following three conditions$:$%
\begin{gather}
h^{\prime}(t)\geq\frac{1}{C}\int_{0}^{t}\int_{\Omega}u_{t}^{2}\,dt;
\label{ft6}\\
h^{\prime}(t)\geq2\Big(\frac{1}{Cp}-p+1\Big)\lambda h(t)\text{, for}%
\;\text{some }\lambda>0;\label{ft7}\\
\frac{1}{2C}\Big(H^{\prime}(t)-H^{\prime}(0)\Big)^{2}\leq H(t)H^{\prime\prime
}(t). \label{ft8}%
\end{gather}

\end{lemma}

\begin{proof}
In fact,%
\begin{align*}
h^{\prime}(t)  &  =\int_{\Omega}u_{t}u\,dx=\int_{\Omega}u(\Delta
_{p}u+f(|u|))\,dx\\
&  \geq\int_{\Omega}\Big(-(p-1)|\nabla u|^{p}+\frac{1}{C}\int_{0}%
^{u}f(|t|)\,dt\Big)dx\\
&  =-\frac{1}{C}\int_{\Omega}\Big(\frac{1}{p}|\nabla u|^{p}-\int_{0}%
^{u}f(|t|)\,dt\Big)dx+\Big(\frac{1}{Cp}-p+1\Big)\int_{\Omega}|\nabla
u|^{p}\,dx.
\end{align*}

Hence,%
\begin{align*}
h^{\prime}(t)  &  \geq-\frac{1}{C}E(u)+\Big(\frac{1}{Cp}-p+1\Big)\int_{\Omega
}|\nabla u|^{p}\,dx\\
&  \geq-\frac{1}{C}E(u)\\
&  =-\frac{1}{C}E(u_{0})+\frac{1}{C}\int_{0}^{t}\int_{\Omega}u_{t}^{2}\,dxdt\\
&  \geq\frac{1}{C}\int_{0}^{t}\int_{\Omega}u_{t}^{2}\,dxdt.
\end{align*}

On the other hand, by the Poincar\'{e} inequality, we have
\begin{align*}
h^{\prime}(t)  &  \geq\Big(\frac{1}{Cp}-p+1\Big)\int_{\Omega}|\nabla
u|^{2}\,dx\\
&  \geq\Big(\frac{1}{Cp}-p+1\Big)\lambda\int_{\Omega}u^{2}\,dx\\
&  =2\Big(\frac{1}{Cp}-p+1\Big)\lambda h(t),
\end{align*}
where $\lambda$ is a suitable positive constant.

We pass now to the proof of (\ref{ft8}). Since
\begin{align*}
H^{\prime}(t)-H^{\prime}(0)  &  =\int_{0}^{t}h^{\prime}(s)\,ds=\int_{0}%
^{t}\int_{\Omega}uu_{t}\,dxdt\\
&  \leq\Big(\int_{0}^{t}\int_{\Omega}u^{2}\,dxdt\Big)^{1/2}\Big(\int_{0}%
^{t}\int_{\Omega}u_{t}^{2}\,dxdt\Big)^{1/2}\\
&  \leq(2H(t))^{1/2}(Ch^{\prime}(t))^{1/2}=(2CH(t)H^{\prime\prime}(t))^{1/2},
\end{align*}
by (\ref{ft6}) we infer that%
\[
H^{\prime}(t)-H^{\prime}(0)=\int_{0}^{t}h^{\prime}(s)\,ds\geq0,
\]
and thus%
\[
\frac{1}{2C}\Big(H^{\prime}(t)-H^{\prime}(0)\Big)^{2}\leq H(t)H^{\prime\prime
}(t).
\]

\end{proof}

\emph{Proof of Theorem \ref{ftthm1}}. Suppose, by \emph{reduction ad
absurdum}, that the solution $u(x,\cdot)$ exists in%
\[
L^{\infty}((0,T);L^{2}(\Omega))\,
\]
for all $T>0$. By (\ref{ft7}),%
\begin{equation}
\lim_{t\rightarrow\infty}H^{\prime}(t)=\lim_{t\rightarrow\infty}h(t)=\infty,
\label{ft10}%
\end{equation}
which yields, for each $\beta\in(0,1/C),$ the existence of a number $T_{0}>0$
such that for all $t>T_{0}$,%
\[
\beta H^{\prime}(t)^{2}\leq\frac{1}{C}\Big(H^{\prime}(t)-H^{\prime
}(0)\Big)^{2}.
\]

Now, by (\ref{ft8}) we obtain%
\[
\beta H^{\prime}(t)^{2}\leq2H(t)H^{\prime\prime}(t).
\]

We will show, by considering the function $G(t)=H(t)^{-q}$, for a suitable
$q>0,$ that the last inequality leads to a contradiction. In fact,%
\begin{align*}
G^{\prime\prime}(t)  &  =qH(t)^{-q-2}\Big((q+1)(H^{\prime}(t))^{2}%
-H(t)H^{\prime\prime}(t)\Big)\\
&  \leq qH(t)^{-q-2}\Big(\frac{2(q+1)}{\beta}-1\Big)H(t)H^{\prime\prime}(t),
\end{align*}
for all $t\geq T_{0},$ so that for $\beta\in\left(  0,1/C\right)  $ and
$q\in\left(  0,1/(2C)-1\right)  $ with $2(q+1)<\beta<1/C,$ the corresponding
function $G(t)$ is concave.

By (\ref{ft10}), $\lim_{t\rightarrow\infty}H(t)=\infty$, whence $\lim
_{t\rightarrow\infty}G(t)=0$. Thus $G$ provides an example of a concave and
strictly positive function which tends to $0$ at infinity, a fact which is not
possible. Consequently $u,$ as a function of $t,$ cannot be in $L^{\infty
}((0,T);L^{2}(\Omega))$ for all $T>0$. The proof of Theorem \ref{ftthm1} is done.

\begin{acknowledgement}
This research is supported by CNCSIS Grant $420/2008$.
\end{acknowledgement}

\end{document}